\documentclass[reqno,12pt,letterpaper]{amsart}
\usepackage{amsmath,amssymb,amsthm,url}
\usepackage[usenames,dvipsnames]{color}
\usepackage[colorlinks=true,linkcolor=Red,citecolor=Green]{hyperref}

\newtheorem{theo}{Theorem}
\newtheorem{prop}{Proposition}[section]

\newtheorem{conj}{Conjecture}
\numberwithin{equation}{section}

\DeclareMathOperator{\Area}{Area}
\DeclareMathOperator{\comp}{comp}

\DeclareMathOperator{\Imag}{Im}
\DeclareMathOperator{\loc}{loc}
\DeclareMathOperator{\PSL}{PSL}
\DeclareMathOperator{\Real}{Re}
\DeclareMathOperator{\Res}{Res}

\DeclareMathOperator{\SL}{SL}

\DeclareMathOperator{\supp}{supp}
\DeclareMathOperator{\Vol}{Vol}
\DeclareMathOperator{\WF}{WF}
\def\WFh{\WF_h}

\title[Microlocal limits of Eisenstein functions
away from the unitarity axis]%
{Microlocal limits of Eisenstein functions\\
away from the unitarity axis}
\author{Semyon Dyatlov}
\email{dyatlov@math.berkeley.edu}
\address{Department of Mathematics, Evans Hall, University of California,
Berkeley, CA 94720, USA}

\begin{document}

\begin{abstract}
We consider a surface $M$ with constant curvature cusp ends and its
Eisenstein functions $E_j(\lambda)$. These are the plane waves
associated to the $j$th cusp and the spectral parameter $\lambda$,
$(\Delta-1/4-\lambda^2)E_j=0$. We prove that as $\Real\lambda\to
\infty$ and $\Imag\lambda\to\nu>0$, $E_j$ converges microlocally to a
certain naturally defined measure decaying exponentially along the
geodesic flow. In particular, for a surface with one cusp and a
sequence of $\lambda$'s corresponding to scattering resonances, we
find the microlocal limit of resonant states with energies away from
the real line. This statement is similar to quantum unique ergodicity
(QUE), which holds in certain other situations; however, the proof uses
only the structure of the infinite ends, not the global properties of
the geodesic flow. As an application, we also show that the scattering
matrix tends to zero in strips separated from the real line.
\end{abstract}

\maketitle

\section{Introduction}

Concentration of eigenfunctions of the Laplacian in phase space dates
back to the papers of Schnirelman~\cite{sch}, Colin de
Verdi\`ere~\cite{cdv}, and Zelditch~\cite{z1}. Their quantum
ergodicity (QE) result states that on a closed Riemannian manifold
whose geodesic flow is ergodic with respect to the Liouville measure,
a density one subsequence of eigenfunctions converges microlocally to this
measure. For manifolds with boundary, QE was proved in a special case
by~G\'erard--Leichtnam~\cite{g-l} and in general by
Zelditch--Zworski~\cite{z-z}. The paper~\cite{g-l} used the
semiclassical defect measure approach taken here.

The papers~\cite{jak,lin,l-s,sound,z2} studied the question for finite
area hyperbolic surfaces, that is hyperbolic quotients with cusps.  In
particular, \cite{z2} established QE for any such surface, if embedded
eigenfunctions are augmented with~\textsl{Eisenstein functions} on the
real line; the latter parametrize continuous spectrum of the Laplacian
arising from the presence of cusps. For the modular surface one has a
stronger statement of quantum unique ergodicity (QUE): any sequence of
Hecke--Maass forms~\cite{lin,sound} or Eisenstein functions on the
real line~\cite{l-s,jak} converges microlocally to the Liouville
measure.  Guillarmou and Naud~\cite{g-n} have recently studied
equidistribution of Eisenstein functions for convex co-compact
hyperbolic manifolds; that is, in the presence of funnels, but not
cusps. Finally, after this paper had been posted as a preprint on
arXiv, an interesting preprint~\cite{p-r-r} appeared that addresses
similar questions to those we study, in particular proving our
Theorem~\ref{t:main}, in the special case of the modular
surface. See~\cite{n,sar,z3} for recent reviews of other results.

The present paper considers an arbitrary surface with cusps and
studies phase space concentration of Eisenstein functions for the
spectral parameter $\lambda$ in the upper half-plane, away from the
real line. We show that for a given cusp and a given limit $\nu>0$ of
$\Imag\lambda$, there is only one limiting measure~--- see
Theorem~\ref{t:main}.  This statement is similar to that of QUE;
however, in contrast with the Q(U)E facts listed above, we do not use
any global properties of the geodesic flow, such as hyperbolicity or
ergodicity. Instead, we represent Eisenstein functions as plane waves;
that is, the sum of `incoming' and `outgoing' waves, where the
`incoming wave' depends only at the structure of the manifold at
infinity.  The main idea of the paper can be summarized as follows:
\textsl{the microlocal limit of a plane wave is obtained by taking the
natural measure corresponding to the `incoming' part of this wave and
propagating it along the geodesic flow}. The key difference from the
case $\Imag\lambda\to 0$ is that the corresponding semiclassical
measures are exponentially decaying, rather than invariant, along the
geodesic flow.

We restrict ourselves to the case of surfaces with exact cusp
ends. However, the method of the proof could potentially be applied to
complete Riemannian manifolds with a variety of infinite ends, or even
to more general self-adjoint semiclassical differential operators, as
long as a notion of plane waves exists. For example, in the case of a
compactly supported metric perturbation of the Euclidean metric on
$\mathbb R^3$, plane waves are solutions to the equation
$$
(\Delta_x-\lambda^2)E(\lambda,\omega;x)=0,\
\omega\in \mathbb S^2,\
\lambda\in \mathbb C,
$$
that have the following form near infinity:
$$
E(\lambda,\omega;x)=e^{-i\lambda \omega\cdot x}+E''(\lambda,\omega;x),
$$
where $E''$ is outgoing (for $\Imag\lambda>0$, this means that it lies
in $L^2$ of the whole space). The limiting measure for
$E(\lambda,\omega;x)$ with $\Real\lambda\to +\infty$, $\Imag\lambda\to
\nu>0$, and $\omega\to \omega_0\in \mathbb S^2$ can be obtained by
propagating forward along the flow the measure $e^{2\nu \omega_0\cdot
x}\,dx$ defined on $\{|x|\gg 1,\ \xi=-\omega_0\}$, similarly to the
definition of the measure $\mu_{j\nu}$ in~\eqref{e:mu-j-nu} below.
 
Our motivation comes from the natural question of quantum ergodicity
of resonant states. These replace eigenfunctions on non-compact
manifolds, and their equidistribution in phase space was studied in
the model of quantized open maps by Nonnenmacher--Rubin~\cite{n1}.  In
a similar setting, Demers and Young~\cite{d-y} have observed that a
conditionally invariant measure for a billiard with a hole is
determined entirely by its behavior in the hole; this is somewhat
parallel to the main idea of our paper presented above.  See also an
interesting physics paper by Keating et al.~\cite{n2}.

As stated in Theorem~\ref{t:resonances}, microlocal convergence for
Eisenstein functions away from the real line yields a microlocal
convergence result for resonant states with complex energies at a
fixed distance from the real line. Although this does not address the
$\Imag\lambda\to 0$ case, satisfied by most resonances (see
Conjecture~\ref{l:qe} below), it seems to be the first result on
microlocal convergence of resonant states of differential operators.

We proceed to a rigorous formulation of the results.  Let $(M,g)$ be a
two-dimensional complete Riemannian manifold with cusp ends; that is,
$M$ is the union of a compact set and finitely many \textsl{cusp
regions} $\mathcal C_1,\dots,\mathcal C_m$, where each $\mathcal C_j$
posesses a system of \textsl{canonical coordinates}
$$
(r,\theta)\in (R,\infty)\times \mathbb S^1,\
\mathbb S^1=\mathbb R/(2\pi \mathbb Z),
$$
with $R$ some constant, such that the metric $g$ on $\mathcal C_j$ has the form
\begin{equation}
  \label{e:cusp-metric}
g=dr^2+e^{-2r}d\theta^2.
\end{equation}
A classical example of such $M$ is a finite area hyperbolic surface
without conic points.  In fact, the present paper applies to finite
area hyperbolic quotients $\Gamma\backslash \mathbb H$ with conic
points as well, as one can get rid of these by passing to a finite
covering space (the cone angles are rational multiples of $\pi$, as
the corresponding elliptic transformation has to generate a discrete
subgroup of~$\PSL(2,\mathbb R)$~--- see for
example~\cite[Chapter~2]{b}).

Let $\Delta$ be the (nonnegative) Laplace--Beltrami operator
corresponding to the metric $g$; this operator is self-adjoint, its
spectrum is contained in $[0,\infty)$, and the spectrum in $[0,1/4)$
consists of finitely many eigenvalues~\cite[Section~1]{mu}.  The
Eisenstein functions
$$
E_j(\lambda),\
j=1,\dots,m,\
\Imag\lambda >0,\
\lambda\not\in (0,i/2]
$$
are unique solutions to the equation%
\footnote{For hyperbolic quotients, is more conventional to use the
parameter $s=1/2-i\lambda$, with $\lambda^2+1/4=s(1-s)$ and the physical region
$\{\Imag\lambda>0\}$ corresponding to $\{\Real s>1/2\}$. We use the
parameter $\lambda$ to emphasize that our argument belongs to general
scattering theory and is applicable to other cases such as the
Euclidean case mentioned above.}
\begin{equation}
  \label{e:main}
(\Delta-1/4-\lambda^2)u=0,\
u\in C^\infty(M),
\end{equation}
that satisfy
\begin{equation}
  \label{e:eisenstein-bd}
u-1_{\mathcal C_j} e^{(1/2-i\lambda)r}\in L^2(M).
\end{equation}
Here $1_{\mathcal C_j}$ is the indicator function of the cusp region
$\mathcal C_j$. To define $L^2(M)$, we use the volume form $\Vol$
induced by $g$. See Section~\ref{s:proofs} for details.

We would like to study in particular the (weak) limit of the measure
$|E_j(\lambda)|^2\,d\Vol$ as $\lambda$ tends to the infinity in a
certain way.  It turns out that it is more natural to study
\textsl{microlocal convergence} of $E_j(\lambda)$ in the sense of
semiclassical defect measures. A definition of these for a compact
manifold can be found in~\cite[Chapter~5]{e-z}; we use semiclassical
notation presented in Section~\ref{s:semiclassical}. Since $M$ is
noncompact and $E_j$ does not lie in $L^2$, we have to insert
compactly supported cutoffs into the definition:
%
%
\noindent\textbf{Definition.} 
\textit{Let $h_n$ be a sequence of positive numbers tending to zero and
$u_n$ be a sequence of functions on $M$ bounded uniformly on $L^2(K)$
for each compact $K\subset M$. (The Eisenstein functions satisfy this
property by~\eqref{e:epp-est-2}.)  We say that the sequence $u_n$
converges microlocally to some Radon measure $\mu$ on $T^*M$, if for
each pseudodifferential operator $A(h)\in\Psi^0$ with principal symbol
$a\in C^\infty(T^*M)$, and each $\chi\in C_0^\infty(M; \mathbb R)$, we
have
$$
\langle A(h_n)\chi u_n,\chi u_n\rangle_{L^2(M)}\to \int_{T^*M} \chi^2a\,d\mu.
$$
The measure $\mu$ is called the semiclassical measure associated to
the sequence $u_n$.}
%
%
In particular, we can take as $A(h)$ the multiplication operator by
$a(z)\in C_0^\infty(M)$:
$$
\int a(z)|u_n|^2\,d\Vol\to \int_{T^*M} a(z)\,d\mu.
$$
In other words, the measure $|u_n|^2\,d\Vol$ converges weakly to the
pushforward of $\mu$ under the projection $\pi:T^*M\to M$.

We list several basic properties of semiclassical measures; we do not
use them in the present paper, but mention them to explain why the
measure $\mu_{j\nu}$ defined below is a reasonable candidate for the
microlocal limit of Eisenstein functions. Assume that $\lambda(h)$ is
a family of complex numbers satisfying%
\footnote{Same methods apply with $\Real\lambda\to -\infty$,
with signs in certain formulas inverted.  The corresponding
semiclassical measures are exponentially increasing along the geodesic
flow and concentrated on the outgoing, rather than incoming, set
$\mathcal A^+_j$.}
\begin{equation}
  \label{e:lambda-h}
\Real\lambda(h)=h^{-1},\
\Imag\lambda(h)\to\nu>0\text{ as }h\to 0.
\end{equation}
Note that by~\eqref{e:main},
\begin{equation}
  \label{e:main-semi}
(h^2\Delta-h^2/4-(1+ih\Imag\lambda(h))^2)E_j(\lambda(h))=0.
\end{equation}
However, $P(h)=h^2\Delta$ is a semiclassical differential operator of
order 2; its principal symbol, which we denote by $p$, is the square
of the norm induced by the metric $g$ on the cotangent
bundle. Therefore, the set $\{p=1\}$ is the cosphere bundle $S^*M$,
consisting of covectors of length $1$; moreover, if $\exp(tV)$ is the
geodesic flow on $T^*M$ and $\exp(tH_p)$ is the Hamiltonian flow of
$p$, then
\begin{equation}
  \label{e:ham-flow}
\exp(tH_p)=\exp(2tV)\text{ on }S^*M.
\end{equation}
Let $h_n$ be any sequence tending to zero such that the sequence
$E_j(\lambda(h_n))$ converges microlocally to some Radon measure $\mu$.  Applying
the methods of proof of~\cite[Section~5.2]{e-z}
to~\eqref{e:main-semi}, one can get the following properties:
\begin{enumerate}
\item $\mu$ is supported on the cosphere bundle: $\mu(T^*M\setminus S^*M)=0$;
\item $\mu$ decays exponentially along the geodesic flow:
for each $A\subset T^*M$,
\begin{equation}
  \label{e:expander}
\mu(\exp(tV)A)=e^{-2\nu t}\mu(A).
\end{equation}
\end{enumerate}
We have exponential decay, rather than invariance under the flow,
in~\eqref{e:expander}, because the imaginary part of the operator
in~\eqref{e:main-semi} is asymptotic to $-2h\nu$. Note that there
exist multiple measures satisfying properties~(1) and~(2) above; each
geodesic emanating directly from some cusp carries such a measure. In
fact, it can be proved%
\footnote{Here is a sketch of the proof. Let $K=\{r\leq R+1\}\subset S^*M$;
since $K$ is compact, we have $\mu(K)=c<\infty$. Then
by~\eqref{e:expander} for each $l\in \mathbb N$,
$\mu(\exp(lV)K)=e^{-2\nu l}c$ and thus the series $\sum_l
\mu(\exp(lV)K)$ converges.  By Borel--Cantelli lemma, we get the
required statement $\mu(S^*M\setminus \mathcal A^-)=0$ if we show that
for each $\rho\in S^*M\setminus \mathcal A^-$, there exist infinitely
many $l\in \mathbb N$ such that $\exp(-lV)\rho\in K$. The latter
follows from the fact that every unit speed backwards geodesic that
leaves $K$ enters some cusp region $\mathcal C_j$; unless this
geodesic lies in $\mathcal A^-$, it will eventually bounce back and
return to $K$, where it will spend an interval of length at least 1,
containing the point $\exp(-lV)\rho$ for some integer $l$.}  that
every Radon measure $\mu$ satisfying properties~(1) and~(2) is
supported on the union $\mathcal A^-$ of the incoming sets $\mathcal
A^-_j$ defined below and thus it can be written as an integral, over
some measure on the circle, of measures supported on geodesics
emanating directly from the cusps.  The main result of the paper is
that there is only one possible semiclassical measure for Eisenstein
functions for fixed $j$ and $\nu>0$:
%
%
\begin{theo}\label{t:main}
Let $h>0$ be a small parameter tending to zero, and assume that
$\lambda(h)$ satisfies~\eqref{e:lambda-h}.  Then for each
$A(h)\in\Psi^0$ with semiclassical principal symbol $a$, and each
$\chi\in C_0^\infty(M;\mathbb R)$, we have as $h\to 0$,
\begin{align}
\label{e:main-1}
\langle A(h)\chi E_j(\lambda(h)),\chi E_j(\lambda(h))\rangle_{L^2(M)}
&\to\int \chi^2 a\,d\mu_{j\nu},\\
\label{e:main-2}
\langle A(h)\chi E_j(\lambda(h)),\chi E_{j'}(\lambda(h))\rangle_{L^2(M)}
&\to 0,\ j\neq j'.
\end{align}
Here $\mu_{j\nu}$ is the measure defined in~\eqref{e:mu-j-nu} below.
\end{theo}
%
%
Together, \eqref{e:main-1} and~\eqref{e:main-2} can be interpreted as
follows: for any $\alpha\in \mathbb C^m$, the linear combination
$\sum_j \alpha_j E_j(\lambda(h))$ converges microlocally to the
measure $\sum_j |\alpha_j|^2 \mu_{j\nu}$.

To construct the measure $\mu_{j\nu}$, we first define the incoming
set $\mathcal A^-_j$ and the outgoing set $\mathcal A^+_j$.  Let
$(r,\theta)$ be the canonical coordinates in the cusp $\mathcal C_j$
and $(r,\theta,p_r,p_\theta)$ be the induced system of coordinates on
$T^*\mathcal C_j$; define
\begin{equation}
  \label{e:a-pm}
\begin{gathered}
\mathcal A^\pm_j=\{\rho\in S^*M\mid \exists t>0: G^{\pm t}
\rho\in\widehat{\mathcal A}^\pm_j\},\\
\widehat{\mathcal A}^{\pm}_j=\{(r,\theta,p_r,p_\theta)
\in T^*\mathcal C_j\mid p_r=\pm 1,\ p_\theta=0\};
\end{gathered}
\end{equation}
In other words, $\mathcal A^+_j$ is the union of all geodesics going
directly into the $j$th cusp and $\mathcal A^-_j$ is the union of all
geodesics emanating directly from it. Note that $\mathcal A^\pm_j$
need not be closed; in fact, for hyperbolic surfaces each of them is
dense in $S^*M$.  The measure $\mu_{j\nu}$ is supported on $\mathcal
A^-_j$ and is constructed as follows: we start with the cylindrical
measure $e^{2\nu r}\,drd\theta$ on $\widehat{\mathcal A}^-_j$ and
propagate it to a measure on the whole $\mathcal A^-_j$ using the
property~(2) of semiclassical measures; the result converges because
$\nu>0$. More formally, for each continuous compactly supported
function $a$ on $T^*M$, we put
\begin{equation}
  \label{e:mu-j-nu}
\int_{T^*M} a\,d\mu_{j\nu}=\lim_{t\to +\infty}e^{-2\nu t}\int_{\mathcal C_j}
e^{2\nu r}(a\circ\exp(tV))(r,\theta,-1,0)\,drd\theta.
\end{equation}
It can be seen directly that~\eqref{e:mu-j-nu} defines a Radon measure
satisfying properties~(1) and~(2) of semiclassical measures.

The proof of Theorem~\ref{t:main} is based on the representation of
Eisenstein functions as plane waves:
$$
E_j(\lambda)=E'_j(\lambda)+E''_j(\lambda),
$$
where $E'_j(\lambda)=1_{\mathcal
C_j}\widetilde\chi(r-R)e^{(1/2-i\lambda)r}$ is the `incoming' and
$E''_j(\lambda)$ the `outgoing' part; since $\lambda$ is in the upper
half-plane, $E''_j$ is bounded in $L^2(M)$ uniformly in $h$.  We now
consider the semiclassical Schr\"odinger propagator
$e^{ith\Delta}=e^{it(h^2\Delta)/h}$; by~\eqref{e:main}, we have
formally
\begin{equation}
  \label{e:meaningless}
e^{ith\Delta}E_j(\lambda)=e^{ith(1/4+\lambda^2)}E_j(\lambda).
\end{equation}
Note that $|e^{ith(1/4+\lambda^2)}|=e^{-2t\Imag\lambda}$ is
exponentially decaying as $t\to+\infty$. For a compactly supported
(that is, having compactly supported Schwartz kernel) $A(h)\in\Psi^0$,
$$
\langle A(h)E_j(\lambda),E_j(\lambda)\rangle=
e^{-4t\Imag\lambda}\langle e^{ith\Delta}A(h)
e^{-ith\Delta}E_j(\lambda),E_j(\lambda)\rangle.
$$
However, as $e^{ith\Delta}$ is unitary and $E''_j$ is bounded in $L^2(M)$, we can
replace $E_j$ by $E'_j$ with a remainder exponentially decaying in $t$:
\begin{equation}
  \label{e:funny}
\langle A(h)E_j(\lambda),E_j(\lambda)\rangle
=e^{-4t\Imag\lambda}\langle e^{ith\Delta}A(h)
e^{-ith\Delta}E'_j(\lambda),E'_j(\lambda)\rangle+\mathcal O(e^{-2t\Imag\lambda}).
\end{equation}
(The remainder is $\mathcal O(e^{-2t\Imag\lambda})$ instead of
$\mathcal O(e^{-4t\Imag\lambda})$, as the intersection of the wavefront set of
$e^{ith\Delta}A(h)e^{-ith\Delta}$ with $S^*M$ lies in $\{r\leq 2t+T\}$
for some constant $T$ depending on the support of $A(h)$ and thus
$e^{ith\Delta}A(h)e^{-ith\Delta}E'_j(\lambda)$ should be of size
$e^{2t\Imag\lambda}$.)  We can now perform an explicit computation
using Egorov's theorem and the formula for $E'_j$ to see that as $h\to
0$, the first term on the right-hand side of~\eqref{e:funny} converges
formally to
$$
e^{-4\nu t}\int_{\mathcal C_j}e^{2\nu r}\widetilde\chi(r-R)^2
(a\circ\exp(2tV))(r,\theta,-1,0)\,drd\theta.
$$
It remains to let $t\to +\infty$ to obtain~\eqref{e:main-1};
\eqref{e:main-2} follows by a similar argument.

There is however a serious gap in the `proof' presented above;
namely, the operator $e^{ith\Delta}$ is only defined on $L^2$ and is
not properly supported. Since the function $E_j(\lambda)$ does not lie
in $L^2$, the equation~\eqref{e:meaningless} does not make any
sense. Similarly, the operator $e^{ith\Delta}A(h)e^{-ith\Delta}$ is
not compactly supported and thus one cannot apply it to
$E'_j(\lambda)$.  To fix this gap, we use cutoffs depending on $t$
and on the support of $A(h)$; see~Proposition~\ref{l:key}.

One could also try to fix the gap discussed in the previous paragraph
by using the propagator
$$
U(t)=\begin{pmatrix}\cos(t\sqrt{\Delta})&\sin(t\sqrt{\Delta})/\sqrt{\Delta}\\
-\sqrt{\Delta}\sin(t\sqrt{\Delta})&\cos(t\sqrt{\Delta})\end{pmatrix}
$$
for the Cauchy problem for the wave equation
\begin{equation}\label{e:we}
(D_t^2-\Delta_x)u(t,x)=0,\
t\in \mathbb R,\
x\in M.
\end{equation}
Indeed, since~\eqref{e:we} has finite speed of
propagation, the elements of the matrix $U(t)$ act $C^\infty(M)\to
C^\infty(M)$.  Define $\sqrt{\lambda^2+1/4}$ to have positive real
part, so that it is equal to $\lambda+\mathcal
O(h)$. By~\eqref{e:main}, the function
$e^{it\sqrt{\lambda^2+1/4}}E_j(\lambda)$ solves~\eqref{e:we} and for
$$
\vec E_j(\lambda)=(1,i\sqrt{\lambda^2+1/4})E_j(\lambda),
$$
we have $U(t)\vec E_j(\lambda)=e^{it\sqrt{\lambda^2+1/4}}\vec E_j(\lambda)$.
One could then try to argue as above, using that 
$U(t)$ is unitary on $\dot H^1(M)\oplus L^2(M)$ to estimate the
contribution of $E''_j(\lambda)$, and using that
$U(-t)(1,i\sqrt{\lambda^2+1/4}) E'_j(\lambda)$ is a Lagrangian
state associated to propagating $\widehat{\mathcal A}^-_j$ by
$\exp(tV)$, to calculate the contribution of $E'_j(\lambda)$.

As an application of Theorem~\ref{t:main}, we derive a bound on the
scattering matrix $S(\lambda)$. For each two cusps $\mathcal C_j$,
$\mathcal C_{j'}$, define $S_{j j'}(\lambda)$ by
$$
E_j|_{\mathcal C_{j'}}(\lambda;r,\theta)=\delta_{jj'} e^{(1/2-i\lambda)r}
+S_{jj'}(\lambda) e^{(1/2+i\lambda)r}+\cdots,
$$
where $(r,\theta)$ are canonical coordinates on $\mathcal C_{j'}$,
$\delta$ is the Kronecker delta, and $\cdots$ denotes the terms
corresponding to terms with $k\neq 0$ in the Fourier series
expansion~\eqref{e:fourier} of $E_j|_{\mathcal C_{j'}}$ in
the~$\theta$ variable.
%
%
\begin{theo}
  \label{t:scattering}
Consider two cusps $\mathcal C_j,\mathcal C_{j'}$ and assume that
$\mu_{j\nu}(\mathcal A^+_{j'})=\emptyset$ (in particular, this is true
for hyperbolic surfaces, as $\mathcal A^+_{j'}\cap \mathcal A^-_j$
consists of countably many geodesics). Then for $\lambda(h)$
satisfying~\eqref{e:lambda-h},
$$
S_{jj'}(\lambda(h))\to 0\text{ as }h\to 0.
$$
\end{theo}
%
%
In other words,
$$
S_{jj'}(\lambda)=o(1),\
0<C^{-1}<\Imag\lambda<C,\
\Real\lambda\to\infty.
$$
This estimate is not always optimal: in the special case of the modular
surface $M=\PSL(2,\mathbb Z)\backslash \mathbb H$, the scattering
coefficient $S(\lambda)$ is related to the Riemann zeta function by
the formula~\cite[Section~2.18]{tit}
$$
S(\lambda)=\sqrt\pi{\zeta(-2i\lambda)\Gamma(-i\lambda)
\over\zeta(1-2i\lambda)\Gamma(1/2-i\lambda)}.
$$
Given that both $\zeta(z)$ and $\zeta^{-1}(z)$ are bounded in every
half-plane $\{\Real z>1+C^{-1}\}$ (either by Dirichlet series or by
Euler product representation), the basic bound on the zeta function in
the critical strip~\cite[(5.1.4)]{tit} gives
\begin{equation}
  \label{e:tit}
|S(\lambda)|=\mathcal O(|\lambda|^{-\min(\Imag\lambda,1/2)-}),\ \Imag\lambda\geq C^{-1}.
\end{equation}
The bound~\eqref{e:tit} is optimal for $\Imag\lambda>1/2$, and no
optimal bounds are known for~$0<\Imag\lambda<1/2$. It would be
interesting to see if semiclassical methods can yield an effective
bound on the scattering coefficients, and compare such bound
to~\eqref{e:tit}.

Finally, we address the question of microlocal convergence of resonant
states. Assume that for some $\lambda$, the matrix $S(\lambda)$ is not
invertible; that is, there exists $\alpha\in\mathbb C^m\setminus
\{0\}$ such that for each $j'$,
\begin{equation}
  \label{e:outgoing}
\sum_j\alpha_j S_{jj'}(\lambda)=0.
\end{equation}
This is equivalent to saying that $-\lambda$ is a \textsl{resonance};
i.e., a pole of the meromorphic continuation of the resolvent
$(\Delta-1/4-\lambda^2)^{-1}$ to the lower half-plane (see for
example~\cite[Section~5]{mu}). Moreover, a resonant state at
$-\lambda$ is given by $\sum_j\alpha_j E_j(\lambda)$.
Theorem~\ref{t:main} immediately implies
%
%
\begin{theo}\label{t:resonances}
Assume that $-\lambda_n$ is a sequence of resonances
satisfying~\eqref{e:lambda-h} for some $h_n\to 0$.  Let $u_{n}$ be a
sequence of corresponding resonant states and assume that it converges microlocally
to some measure $\mu$.  Then $\mu$ is a linear combination of the
measures $\mu_{1\nu},\dots,\mu_{m\nu}$ defined by~\eqref{e:mu-j-nu}.
\end{theo}
%
%
The fact that semiclassical measures for resonant states are
exponentially decaying along the geodesic flow is parallel
to~\cite[Theorem~4]{n-z}; a similar fact has been obtained in the
setting of quantized open maps in~\cite{n1}. However, the
concentration statement~\cite[(1.15)]{n-z} is vacuous in our case, as
the set $\Gamma_E^-$ from~\cite{n-z} (not $\Gamma_E^+$, as
$\Real(-\lambda)<0$) is the whole cosphere bundle. In fact, \cite{n-z}
heavily use the fact that resonant states are outgoing, while
Eisenstein functions studied in the present paper need not satisfy the
outgoing condition (which in our case is~\eqref{e:outgoing}).

For surfaces with only one cusp, resonant states away from the real
line have to converge microlocally to a single measure; however, we do
not address the question of the behavior of the kernel of $S(\lambda)$
for multiple cusps. For other types of infinite ends (such as the
convex cocompact case), the scattering matrix $S(\lambda)$ is replaced
by an operator acting on a certain infinite dimensional Hilbert
space. The distribution of resonances in strips of fixed size and the
behavior of the kernel of $S(\lambda)$ is controlled by
the~\textsl{trapping} phenomenon in the compact part of our manifold,
while the restriction on the set of possible semiclassical measures
provided in this paper only uses behavior at infinity.

Finally, it would be natural to ask an analogue of the quantum
ergodicity question: where do most resonant states microlocally
converge if we drop the restriction $\Imag\lambda\to\nu>0$?
Theorem~\ref{t:resonances} does not provide the answer because most
resonances are located $o(1)$ close to the real line. To make this
last statement precise, we assume that there are only finitely many
embedded eigenvalues (which is true under certain genericity
assumptions~--- see~\cite[Th\'eor\`eme~7]{cdvG} and~\cite{p-s}),
let~$\Res$ be the set of resonances, counted with multiplicities, and
recall the Weyl law~\cite[(0.5)]{mu}:
$$
|\Res\cap\{|\lambda|\leq h^{-1}\}|={\Area(M)\over 2\pi}h^{-2}(1+o(1)).
$$
We claim that for each $\varepsilon>0$,
\begin{equation}
  \label{e:strip}
|\Res\cap\{|\lambda|\leq h^{-1}, \Imag\lambda<-\varepsilon\}|=o(h^{-2}).
\end{equation}
The proof of~\eqref{e:strip} is based on~\cite[Corollary~3.29]{mu}:
\begin{equation}
  \label{e:carleman}
\sum_{\lambda\in\Res} {|\Imag\lambda|\over |\lambda|^2}<\infty.
\end{equation}
If~\eqref{e:strip} is false, then there exist $\varepsilon,\delta>0$
and a sequence $h_j\to 0$ such that
$$
|\Res\cap\{|\lambda|\leq h_j^{-1},\ \Imag\lambda<-\varepsilon\}|\geq 2\delta h_j^{-2}.
$$
However, by the upper bound provided by the Weyl law, there exists a
constant $C_0$ such that
\begin{equation}
  \label{e:strip-2}
|\Res\cap\{C_0^{-1}h_j^{-1}\leq |\lambda|\leq h_j^{-1},\
\Imag\lambda<-\varepsilon\}|\geq \delta h_j^{-2}.
\end{equation}
We pass to a subsequence of $h_j$ such that $h_j/h_{j+1}>C_0$. Then
the sets of resonances counted in~\eqref{e:strip-2} for different $j$
do not intersect each other, and the sum
of~$|\Imag\lambda|/|\lambda|^2$ over each of these sets is bounded
from below by $\varepsilon\delta$, contradicting~\eqref{e:carleman}.

By~\eqref{e:strip}, a density one subsequence of resonances converges
to the real line; the corresponding semiclassical measures are
invariant with respect to the geodesic flow and a natural candidate is
the Liouville measure:
%
%
\begin{conj}\label{l:qe}
Assume that $M$ is a surface with cusp ends whose geodesic flow is
ergodic with respect to the Liouville measure $\mu_0$. Then there
exists a density one subsequence of resonant states in any strip
$\{\Imag\lambda>-C\}$ converging microlocally to $\mu_0$.
\end{conj}
%
%
The proof of~\eqref{e:strip} and Conjecture~\ref{l:qe} have been
suggested by Maciej Zworski.

The paper is organized as follows. In Section~\ref{s:semiclassical},
we review some notation and facts from semiclassical analysis. In
Section~\ref{s:proofs}, we present basic facts about Eisenstein
functions and prove Theorems~\ref{t:main} and~\ref{t:scattering}.
Finally, in Section~\ref{s:hyperbolic}, we consider the special case
of finite area hyperbolic surfaces and describe the canonical measures
$\mu_{j\nu}$ from~\eqref{e:mu-j-nu} via the action of the fundamental
group of $M$; we also prove Theorem~\ref{t:main} in this case for
$\Imag\lambda>1/2$ using the classical definition of Eisenstein
functions as series.

\section{Semiclassical preliminaries}\label{s:semiclassical}

In this section, we briefly review the portions of semiclassical
analysis used below; the reader is referred to~\cite{e-z} for a
detailed account on the subject.

We assume that $h>0$ is a parameter, the smallness of which is implied
in all statements below. Consider the algebra $\Psi^s(\mathbb R^d)$ of
pseudodifferential operators with symbols in the class $S^s(\mathbb
R^d)$, defined as follows: a function $a(x,\xi;h)$ smooth in
$(x,\xi)\in \mathbb R^{2d}$ lies in this class if and only if for each
compact set $K\subset \mathbb R^d$ and each multiindices
$\alpha,\beta$, there exists a constant $C_{\alpha\beta K}$ such that
$$
\sup_{(x,\xi)\in K\times \mathbb R^d}|\partial_\alpha^x\partial_\beta^\xi
a(x,\xi;h)|\leq C_{\alpha\beta K}\langle\xi\rangle^{s-|\beta|}.
$$
The only difference with the invariant symbol classes studied
in~\cite[Section~9.3]{e-z} is that we do not require uniform bounds as
$x\to\infty$. However, this does not matter in our situation, as we
will mostly use \textsl{compactly supported} operators; e.g. those
operators whose Schwartz kernels are compactly supported in $\mathbb
R^d\times \mathbb R^d$.  As in~\cite[Section~13.2]{e-z}, we can define
the algebra $\Psi^s(M)$ for any manifold $M$. The compactly supported
elements of $\Psi^s(M)$ act $H^t_{\hbar,\loc}(M)\to
H^{t-s}_{\hbar,\comp}(M)$ with norm $\mathcal O(1)$ as $h\to 0$, where
$H^t_{\hbar,\loc}$ and $H^{t-s}_{\hbar,\comp}$ are semiclassical
Sobolev spaces.

To avoid discussion of simultaneous behavior of symbols as
$\xi\to\infty$ and $h\to 0$, we further require that the symbols of
elements of $\Psi^s$ are \textsl{classical}, in the sense that they
posess an asymptotic expansion in powers of $h$, with the term next to
$h^k$ lying in $S^{s-k}$ (see~\cite[Section~2.1]{zeeman} for details).
Following~\cite[Section~2]{v}, we introduce the fiber-radial
compactification $\overline T^*M$ of the cotangent bundle.
(The use of $\overline T^*M$ slightly simplifies the statement of
Proposition~\ref{l:elliptic} below. However, as all other microlocal
analysis is happening in a compact subset of $T^*M$, one could easily
avoid fiber-radial compactification if needed.)
Each $A\in\Psi^s$ has an invariantly defined (semiclassical) principal
symbol $\sigma(A)=a\in C^\infty(T^*M)$, and $\langle\xi\rangle^{-s}a$
extends to a smooth function on $\overline T^*M$. We then define the
characteristic set of $A$ as $\{\langle\xi\rangle^{-s}a= 0\}\subset
\overline T^*M$ and say that $A$ is elliptic on some $U\subset
\overline T^*M$, if $U$ does not intersect the characteristic set of
$A$.

We use the definition of semiclassical wavefront set $\WFh(A)\subset
\overline T^*M$ for $A\in\Psi^s(M)$ found in~\cite[Section~2]{v}
or~\cite[Section~2.1]{zeeman}.  The wavefront set of $A$ is empty if
and only if $A$ lies in the algebra $h^\infty\Psi^{-\infty}(M)$ of
smoothing operators such that each of $C^\infty(M\times M)$ seminorms
of their Schwartz kernels decays faster than any power of $h$.  For
$A,B\in\Psi^s(M)$, we say that $A=B$ microlocally on some open
$U\subset \overline T^*M$, if $\WFh(A-B)\cap U=\emptyset$. Also, we
say that $A\in\Psi^s(M)$ is compactly microlocalized, if $\WFh(A)$
does not intersect the fiber infinity $\partial(\overline T^*M)$; in
this case, $A\in\Psi^s(M)$ for all $s\in \mathbb R$.

We now recall several fundamental facts from semiclassical analysis:
%
%
\begin{prop}\label{l:elliptic} (Elliptic estimate)
Let $P\in\Psi^s(M)$, be properly supported, $A\in\Psi^t(M)$, $t\leq
s$, be compactly supported, and assume that $P$ is elliptic on
$\WFh(A)\subset \overline T^*M$.  Then there exists a compact set
$K\subset M$ and a constant $C$ such that for each $u\in
H^s_{\hbar,\loc}(M)$,
$$
\|Au\|_{L^2(M)}\leq  C\|Pu\|_{L^2(K)}+\mathcal O(h^\infty)\|u\|_{L^2(K)}.
$$
\end{prop}
%
%
%
%
\begin{prop}\label{l:egorov}
Assume that $M$ is a manifold with prescribed volume form and
$P(h)\in\Psi^s$ is a properly supported self-adjoint operator on
$L^2(M)$ with principal symbol $p\in C^\infty(T^*M;\mathbb R)$.
Let
$$
e^{itP(h)/h}:L^2(M)\to L^2(M),\ t\in \mathbb R,
$$
be the corresponding Schr\"odinger propagator, defined by means of
spectral theory; it is a unitary operator.  Let also $\exp(tH_p)$ be
the Hamiltonian flow of $p$ on $T^*M$. Then (with constants below
depending on $t$):

1. (Microlocalization) The operator $e^{itP(h)/h}$ is microlocalized
on the graph of $\exp(-tH_p)$ in the following sense: if
$A,B\in\Psi^0(M)$ are compactly supported and $B$ is compactly
microlocalized, and
\begin{equation}
  \label{e:mic-req}
\exp(tH_p)(\WFh(A))\cap\WFh(B)=\emptyset,
\end{equation}
then $Ae^{itP(h)/h}B=\mathcal O(h^\infty)_{L^2\to L^2}$.

2. (Egorov's Theorem) Let $A\in\Psi^0(M)$ be compactly supported and
compactly microlocalized, with principal symbol $a$.  Then there
exists a compactly supported and compactly microlocalized operator
$A_t\in\Psi^0(M)$ such that
\begin{equation}
  \label{e:egorov}
e^{itP(h)/h}Ae^{-itP(h)/h}=A_t+\mathcal O(h^\infty)_{L^2(M)\to L^2(M)}.
\end{equation}
Moreover, $\WFh(A_t)\subset\exp(-tH_p)(\WFh(A))$, and the principal
symbol of $A_t$ is $a\circ\exp(tH_p)$.
\end{prop}
%
%
A proof of Proposition~\ref{l:elliptic} in the closely related microlocal
case is given in~\cite[Theorem~18.1.24']{ho3}; see for example~\cite[Section~2.2]{zeeman}
for the semiclassical case. For Proposition~\ref{l:egorov},
see~\cite[Theorem~11.1]{e-z} or~\cite[Proposition~2.6]{zeeman}.

\section{Proofs}\label{s:proofs}

We start by studying the equation~\eqref{e:main} in some cusp
$\mathcal C_j$.  Consider the Fourier series
\begin{equation}
  \label{e:fourier}
u|_{\mathcal C_j}(r,\theta)=\sum_{k\in \mathbb Z} u_k^{j}(r)e^{ik\theta}.
\end{equation}
By~\eqref{e:cusp-metric}, \eqref{e:main} takes the form
\begin{equation}
  \label{e:separated}
[(D_r+i/2)^2+k^2 e^{2r}-\lambda^2]u_k^{j}(r)=0,\
k\in \mathbb Z.
\end{equation}
For $k=0$, \eqref{e:separated} is a constant-coefficient ODE and we have
\begin{equation}
  \label{e:u0}
u_0^{j}(r)=u_+^{j}e^{(1/2+i\lambda)r}+u_-^{j}e^{(1/2-i\lambda)r},
\end{equation}
for some constants $u_\pm^j$.

Now, we extend the function $r$ smoothly from the union of all
$\mathcal C_j$ to the whole $M$ so that $r\leq R$ outside of the cusp
regions.  As before, let $1_{\mathcal C_j}$ be the indicator function
of the cusp $\mathcal C_j$.  Finally, fix a cutoff function
$\widetilde\chi\in C^\infty(\mathbb R;[0,1])$ such that
$\supp\widetilde\chi\subset (0,\infty)$ and
$\supp(1-\widetilde\chi)\subset (-\infty, 1)$.

Take a cusp $\mathcal C_j$ and define the `incoming' part of the Eisenstein function by
\begin{equation}
  \label{e:e-prime-j}
E'_j(\lambda)=1_{\mathcal C_j}\cdot\widetilde\chi(r-R)e^{(1/2-i\lambda)r}\in C^\infty(M).
\end{equation}
Then
$$
F_j(\lambda)=(\Delta-1/4-\lambda^2)E'_j(\lambda)
=1_{\mathcal C_j}[\Delta,\widetilde\chi(r-R)]e^{(1/2-i\lambda)r}\in C_0^\infty(M).
$$
Assume that for some constant $C_0$, we have
\begin{equation}
  \label{e:lambda-cond}
C_0^{-1}\leq \Imag\lambda\leq C_0,\
\Real\lambda>1.
\end{equation}
Since $\Delta$ is self-adjoint, the resolvent
$$
(\Delta-1/4-\lambda^2)^{-1}:L^2(M)\to L^2(M)
$$
is well-defined and the only solution to~\eqref{e:main}
satisfying~\eqref{e:eisenstein-bd} is given by~\cite[Section~3]{mu}
\begin{equation}
  \label{e:eisenstein-defi}
\begin{gathered}
E_j(\lambda)=E'_j(\lambda)+E''_j(\lambda),\\
E''_j(\lambda)=-(\Delta-1/4-\lambda^2)^{-1}F_j(\lambda).
\end{gathered}
\end{equation}
We can estimate $E''_j$ uniformly in $L^2(M)$:
%
%
\begin{prop}\label{l:bb}
There exists a constant $C$ such that for each $\lambda$ satisfying~\eqref{e:lambda-cond},
\begin{equation}
  \label{e:epp-est}
\|E''_j(\lambda)\|_{L^2(M)}\leq C.
\end{equation}
It follows that for each compact $K\subset M$, there exists a constant
$C_K$ such that for each $\lambda$ satisfying~\eqref{e:lambda-cond},
\begin{equation}
  \label{e:epp-est-2}
\|E_j(\lambda)\|_{L^2(K)}\leq C_K.
\end{equation}
\end{prop}
\begin{proof}
It follows from the definition of $F_j$ that $\|F_j\|_{L^2}=\mathcal O(|\lambda|)$.
However, since $\Delta$ is self-adjoint,
$$
\|(\Delta-1/4-\lambda^2)^{-1}\|_{L^2\to L^2}\leq {1\over|\Imag(\lambda^2)|}=\mathcal O(|\lambda|^{-1})
$$
and~\eqref{e:epp-est} follows. Next, \eqref{e:epp-est-2} follows
from~\eqref{e:epp-est} and the fact that $\|E'_j\|_{L^2(K)}\leq C_K$.
\end{proof}
%
%
Also, $E_j$ is microlocalized on the cosphere bundle $S^*M$:
%
%
\begin{prop}\label{l:ell-e}
Assume that $A(h)\in\Psi^s(M)$, $s\leq 2$, is compactly supported and
$\WFh(A)\cap S^*M=\emptyset$. Let $h=(\Real\lambda)^{-1}$. Then
$\|A(h)E_j(\lambda)\|_{L^2(M)}=\mathcal O(h^\infty)$.
\end{prop}
\begin{proof}
Follows from Proposition~\ref{l:elliptic} applied
to~\eqref{e:main-semi} and~\eqref{e:epp-est-2}.
\end{proof}
%
%
We now prove the key technical estimate, approximating $E_j(\lambda)$
on a fixed compact set by the result of propagating the appropriately
cut off `incoming wave' $E_j'(\lambda)$:
%
%
\begin{prop}\label{l:key}
Assume that $\lambda$ satisfies~\eqref{e:lambda-cond}. As before,
let $h=(\Real\lambda)^{-1}$. Then for each $T>R$, $t>0$,
$$
\begin{gathered}
\|\widetilde\chi(T-r)(E_j(\lambda)-e^{ith(1/4+\lambda^2)}
e^{-ith\Delta}\widetilde\chi(T+2t+1-r)E'_j(\lambda))\|_{L^2(M)}\\
=\mathcal O_{t,T}(h^\infty)+\mathcal O(e^{-2t\Imag\lambda}).
\end{gathered}
$$
Here the $\mathcal O_{t,T}$ notation means that the constants in $\mathcal O(\cdot)$
depend on $t$ and $T$. The constant in $\mathcal O(e^{-2t\Imag\lambda})$ is
independent of $T,t,h$. Note also that $\widetilde\chi(T-r)\in
C_0^\infty(M)$.
\end{prop}
\begin{proof}
First of all, we have by~\eqref{e:epp-est},
$$
\|e^{ith(1/4+\lambda^2)}e^{-ith\Delta}\widetilde\chi(T+2t+1-r)E''_j(\lambda)\|_{L^2}
\leq e^{-2t\Imag\lambda}\|E''_j(\lambda)\|_{L^2}=\mathcal O(e^{-2t\Imag\lambda}).
$$
Therefore, it suffices to prove that
\begin{equation}
  \label{e:key-int}
\|u_t\|_{L^2}=\mathcal O_{t,T}(h^\infty),
\end{equation}
where for $0\leq s\leq t$, we define
$$
u_s=\widetilde\chi(T-r)(E_j(\lambda)-e^{ish(1/4+\lambda^2)}
e^{-ish\Delta}\widetilde\chi(T+2t+1-r)E_j(\lambda)).
$$
Since $\widetilde\chi(T-r)\widetilde\chi(T+2t+1-r)=\widetilde\chi(T-r)$,
we have $u_0=0$; next,
$$
\begin{gathered}
D_s u_s=\widetilde\chi(T-r)e^{ish(1/4+\lambda^2)}e^{-ish\Delta}
h(\Delta-1/4-\lambda^2)\widetilde\chi(T+2t+1-r)E_j(\lambda)\\
=\widetilde\chi(T-r)e^{ish(1/4+\lambda^2)}e^{-ish\Delta}
h[\Delta,\widetilde\chi(T+2t+1-r)]E_j(\lambda).
\end{gathered}
$$
Let $X(h)\in\Psi^0$ be compactly supported and compactly
microlocalized in a small neighborhood of the cosphere bundle $S^*M$,
but equal to the identity microlocally near $\{r\leq T+2t+1\}\cap
S^*M$. Then by Proposition~\ref{l:ell-e}
\begin{equation}
  \label{e:key-1}
\|h[\Delta,\widetilde\chi(T+2t+1-r)](1-X(h))E_j(\lambda)\|_{L^2}=\mathcal O_{t,T}(h^\infty).
\end{equation}
Now, by part~1 of Proposition~\ref{l:egorov},
\begin{equation}
  \label{e:key-3}
\|\widetilde\chi(T-r)e^{-ish\Delta}h[\Delta,\widetilde\chi(T+2t+1-r)]X(h)\|_{L^2\to L^2}=\mathcal O_{t,T}(h^\infty).
\end{equation}
To verify~\eqref{e:mic-req}, we note that each point
$\rho\in\WFh(h[\Delta,\widetilde\chi(T+2t+1-r)]X(h))$ lies close to
the cosphere bundle $S^*M$ (depending on the choice of $X(h)$) and
inside $\{r>T+2t\}$; therefore, by~\eqref{e:ham-flow}, the curve
$\exp([0,t]H_p)\rho$ lies in $\mathcal C_j\cap \{r>T\}$, and thus does
not intersect the support of $\widetilde\chi(T-r)$. Here we use the
fact that in each cusp region $\mathcal C_j$, the derivative of the
function $r$ along the geodesic flow is bounded by 1, which can be
verified directly using~\eqref{e:cusp-metric}.

Since the operator in~\eqref{e:key-3} is compactly supported, by~\eqref{e:epp-est-2} we get
\begin{equation}
  \label{e:key-2}
\|\widetilde\chi(T-r)e^{-ish\Delta}h[\Delta,\widetilde\chi(T+2t+1-r)]X(h)E_j(\lambda)\|_{L^2}=\mathcal O_{t,T}(h^\infty).
\end{equation}
Combining~\eqref{e:key-1} with~\eqref{e:key-2}, we arrive to
$$
\|d_su_s\|_{L^2}=\mathcal O_{t,T}(h^\infty);
$$
integrating this from $0$ to $t$, we get~\eqref{e:key-int}.
\end{proof}
%
%
Armed with Proposition~\ref{l:key}, we can make rigorous the `proof' of
the main theorem given in the introduction: 
%
%
\begin{proof}[Proof of Theorem~\ref{t:main}]
Let $A\in\Psi^0$ be compactly supported with principal symbol $a$; it
suffices to prove~\eqref{e:main-1} and~\eqref{e:main-2} without the
cutoff $\chi$. We may assume that $A$ is compactly microlocalized;
indeed, if $\WFh(A)\cap S^*M=\emptyset$, then $\langle
AE_j(\lambda),E_{j'}(\lambda)\rangle=\mathcal O(h^\infty)$ by
Proposition~\ref{l:ell-e} and~\eqref{e:epp-est-2}. In fact, we may
assume that $\WFh(A)$ is contained in a small neighborhood of $S^*M$.

Fix $T>R$ such that $A$ is supported in $\{r<T-1\}$, so that
$A=\widetilde\chi(T-r)A\widetilde\chi(T-r)$.  By
Proposition~\ref{l:key},
\begin{equation}
  \label{e:est-1}
\begin{gathered}
\widetilde\chi(T-r)E_j(\lambda)=e^{ith(1/4+\lambda^2)}\widetilde\chi(T-r)
e^{-ith\Delta}\widetilde\chi(T+2t+1-r)E'_j(\lambda)\\
+\mathcal O_{t}(h^\infty)_{L^2}+\mathcal O(e^{-2t\Imag\lambda})_{L^2}.
\end{gathered}
\end{equation}
Take some $j,j'$; substituting~\eqref{e:est-1} into the expression
$$
\langle AE_j(\lambda),E_{j'}(\lambda)\rangle
=\langle A\widetilde\chi(T-r)E_j(\lambda),\widetilde\chi(T-r)E_{j'}(\lambda)\rangle,
$$
and using the fact that the left-hand side of~\eqref{e:est-1} is $\mathcal O(1)$
in $L^2$ by~\eqref{e:epp-est-2}, we get
\begin{equation}
  \label{e:conjj}
\begin{gathered}
|\langle AE_j(\lambda),E_{j'}(\lambda)\rangle
-e^{-4t\Imag\lambda}\langle \widetilde A_t E'_j(\lambda),E'_{j'}(\lambda)\rangle|
\leq Ce^{-\nu t}+\mathcal O_t(h^\infty),\\
\widetilde A_t=\widetilde\chi(T+2t+1-r)e^{ith\Delta}Ae^{-ith\Delta}\widetilde\chi(T+2t+1-r),
\end{gathered}
\end{equation}
where we use that $\Imag\lambda>\nu/2$ for $h$ small enough.
Here $C$ is a constant depending on $A$ and $T$, but not on $t$ or $h$.
Therefore,
$$
\lim_{t\to +\infty} \limsup_{h\to 0}|\langle AE_j(\lambda),E_{j'}(\lambda)\rangle
-e^{-4t\Imag\lambda}\langle\widetilde A_t E'_j(\lambda),E'_{j'}(\lambda)\rangle|=0,
$$
and thus
\begin{equation}
  \label{e:conj}
\begin{gathered}
\lim_{h\to 0}\langle AE_j(\lambda),E_{j'}(\lambda)\rangle
=\lim_{t\to +\infty}e^{-4t\Imag\lambda}\lim_{h\to 0}\langle
\widetilde A_tE'_j(\lambda),E'_{j'}(\lambda)\rangle,
\end{gathered}
\end{equation}
provided that the double limit on the right-hand side exists.
To compute this limit, let $A_t\in\Psi^0$ be the compactly supported operator from
part~2 of Proposition~\ref{l:egorov}, with $P(h)=h^2\Delta$.  Then
$\WFh(A_t)\subset \{r\leq T+2t\}$; therefore,
$$
A_t=\widetilde\chi(T+2t+1-r)A_t\widetilde\chi(T+2t+1-r)+\mathcal O_t(h^\infty)_{L^2\to L^2};
$$
by~\eqref{e:egorov}, $\widetilde A_t= A_t+\mathcal O_t(h^\infty)_{L^2\to
L^2}$. Since both $A_t$ and $\widetilde A_t$ are compactly supported,
we can replace $\widetilde A_t$ by $A_t$ in~\eqref{e:conj}.
By~\eqref{e:ham-flow}, the principal symbol of $A_t$ on $S^*M$ is
$a_t=a\circ\exp(2tV)$.

For $j\neq j'$, the supports of the functions $E'_j(\lambda)$ and
$E'_{j'}(\lambda)$ do not intersect (as they lie in different cusp
regions); since $A_t$ is pseudodifferential and compactly supported,
we get
$$
\langle A_tE'_j(\lambda),E'_{j'}(\lambda)\rangle=\mathcal O_{t}(h^\infty)
$$
and~\eqref{e:main-2} follows.

We now assume that $j=j'$. We can use the
definition~\eqref{e:e-prime-j} of $E'_j$, the definition of
semiclassical quantization, and the method of stationary phase to get
for each $t$,
\begin{equation}\label{e:lmul}
A_tE'_j(\lambda)=1_{\mathcal C_j}(a\circ\exp(2tV))(r,\theta,-1,0)
\widetilde\chi(r-R)e^{(1/2+\Imag\lambda)r}e^{-ir/h}+\mathcal O_t(h)_{L^2_{\comp}}.
\end{equation}
Indeed, restricting to the cusp $\mathcal C_j$, in local coordinates
$(r,\theta)$ (we will also need to restrict to a topologically trivial subset
of the circle, where $\theta\in \mathbb R$ gives a valid coordinate)
the left-hand side of~\eqref{e:lmul} becomes
$$
(2\pi h)^{-2}\int e^{{i\over h}((r-r',\theta-\theta')\cdot (p_r,p_\theta)-r')}
a_t(r,\theta;p_r,p_\theta)\widetilde \chi(r'-R)e^{(1/2+\Imag\lambda)r'}\,dp_r
dp_\theta dr'd\theta'+\mathcal O(h).
$$
The stationary points of the phase $\Phi=(r-r',\theta-\theta')\cdot (p_r,p_\theta)-r'$
are given by $r'=r,\theta'=\theta,p_r=-1,p_\theta=0$ and at these points,
$\Phi$ takes the value $-r$ and its Hessian has determinant $1$ and signature $0$.
Applying the method of stationary phase, we get~\eqref{e:lmul}.

We now multiply~\eqref{e:lmul} by $E'_j(\lambda)$ and integrate, remembering
that the volume form on $\mathcal C_j$ is given by $e^{-r}\,drd\theta$
and $\Imag\lambda\to \nu$ as $h\to 0$:
\begin{equation}
  \label{e:est-2}
\lim_{h\to 0}\langle A_tE'_j(\lambda),E'_j(\lambda)\rangle
=\int_{\mathcal C_j}\widetilde\chi(r-R)^2 e^{2\nu r}(a\circ\exp(2tV))(r,\theta,-1,0)\,drd\theta.
\end{equation}
Here $(r,\theta,p_r,p_\theta)$ are the coordinates on $T^* \mathcal
C_j$ induced by the coordinate system $(r,\theta)$ on $\mathcal C_j$.
Letting $t\to +\infty$ and recalling~\eqref{e:mu-j-nu}, we get
from~\eqref{e:est-2}
$$
\lim_{h\to 0}\langle AE_j(\lambda),E_j(\lambda)\rangle=\int_{S^*M} a\,d\mu_{j\nu},
$$
which proves~\eqref{e:main-1}.
\end{proof}
%
%
We can explain~\eqref{e:lmul} using the theory of semiclassical
Lagrangian distributions (see~\cite[Chapter~6]{g-s}
or~\cite[Section~2.3]{svn} for a detailed account,
and~\cite[Section~25.1]{ho4} or~\cite[Chapter~11]{gr-s} for the
closely related microlocal case) as follows.  By~\eqref{e:e-prime-j},
the function
$$
E'_j(\lambda)=1_{\mathcal C_j}\cdot\widetilde\chi(r-R)e^{(1/2+\Imag\lambda)r}e^{-ir/h}
$$
is a Lagrangian distribution associated to the Lagrangian
$\widehat{\mathcal A}^-_j$ from~\eqref{e:a-pm}, with the principal
symbol $\widetilde\chi(r-R)e^{(1/2+\Imag\lambda)r}$. Since $A_t$ is
pseudodifferential and compactly supported, $A_tE'_j(\lambda)$ is also
a Lagrangian distribution associated to $\widehat{\mathcal A}^-_j$,
and its principal symbol is the product of the principal symbol of
$E'_j(\lambda)$ and the restriction of the principal symbol of $A_t$
to $\widehat{\mathcal A}^-_j$, proving~\eqref{e:lmul}.

Finally, we estimate the scattering coefficient by the mass of Eisenstein
series on the outgoing set:
%
%
\begin{proof}[Proof of Theorem~\ref{t:scattering}]
Fix a cusp $\mathcal C_{j'}$ and take compactly supported and
compactly microlocalized $A\in\Psi^0((R,R+1))$ such that
$\WFh(A)\subset \{p_r>0\}$, and the principal symbol $a(r,p_r)$ of $A$
satisfies $a(R+1/2,1)\neq 0$.  Let $\chi\in C_0^\infty(\mathbb R)$
have $\chi(0)=1$.  Denote $u(h)=E_j(\lambda(h))$ and recall that
$S_{jj'}(\lambda(h))=u_+^{j'}(h)$ is defined by~\eqref{e:u0}.  We then
have for each $\delta>0$,
\begin{equation}\label{e:u-minus}
|u_+^{j'}(h)| = \mathcal O(1) \|Au_0^{j'}(h)\|_{L^2}+\mathcal O(h^\infty)
= \mathcal O(1) \|A_\delta u(h)\|_{L^2(M)}+\mathcal O(h^\infty),
\end{equation}
uniformly in $\delta$, where $A_\delta=\chi(hD_\theta/\delta)A$ is a
pseudodifferential operator supported in $\mathcal C_{j'}$.  However,
by Theorem~\ref{t:main} we have as $h\to 0$
$$
\|A_\delta u(h)\|_{L^2(M)}^2
=\langle A_\delta^*A_\delta u(h),u(h)\rangle_{L^2(M)}
\to \int_{S^* \mathcal C_{j'}} |\chi(hp_\theta/\delta)a(r,p_r)|^2\,d\mu_{j\nu}.
$$
By our assumption, $\mu_{j\nu}(\widehat{\mathcal A}^+_{j'})=0$; therefore,
$$
\lim_{\delta\to 0}\lim_{h\to 0} \|A_\delta u(h)\|_{L^2(M)}=
\int_{\widehat{\mathcal A}^+_{j'}} |a(r,p_r)|^2\,d\mu_{j\nu}=0
$$
and we are done by~\eqref{e:u-minus}.
\end{proof}
%
%

\section{Hyperbolic surfaces}\label{s:hyperbolic}

In this section, we consider the special case $M=\Gamma\backslash
\mathbb H$, where $\mathbb H$ is the Poincar\'e half-plane model
of the hyperbolic plane and $\Gamma\subset\PSL(2,\mathbb R)$ is a
Fuchsian group of the first kind, so that $M$ is a finite area
hyperbolic surface. Denote by~$\pi_\Gamma:\mathbb H\to M$ the
projection map. The conformal boundary $\partial \mathbb H=\mathbb
R\cup\{\infty\}$ is a circle, as can be seen by using the Poincar\'e
ball model. This section is not used anywhere else in the paper and is
provided as a quick reference for readers familiar with the theory of
hyperbolic surfaces.

We first find an interpretation of~\eqref{e:expander} in terms of the
group action; this is parallel to the representation of measures
invariant under the Hamiltonian flow in Patterson--Sullivan theory
(see for example~\cite[Section~14.2]{b}). We parametrize the cosphere
bundle $S^*\mathbb H$ by
$$
\mathcal T:(\partial \mathbb H\times \partial \mathbb H)_\Delta\times \mathbb R\to S^*\mathbb H,
$$
where $(\partial \mathbb H\times \partial \mathbb H)_\Delta$ is the
Cartesian square of the circle $\partial \mathbb H$ minus the
diagonal. The map $\mathcal T$ is defined as follows: take
$(q_1,q_2)\in (\partial \mathbb H\times \partial \mathbb H)_\Delta$
and let $\gamma_{q_1q_2}(t)$ be the unique unit speed geodesic
(that is, a semicircle in the half-plane model) going
from $q_1$ to $q_2$, parametrized so that $\gamma(0)$ is the point of
$\gamma$ closest to $i\in \mathbb H$. We put
$$
\mathcal T(q_1,q_2,t)=(\gamma_{q_1q_2}(t),\dot\gamma_{q_1q_2}(t)).
$$
Now, consider a Radon measure $\mu$ on $S^* M$ satisfying~\eqref{e:expander}.
We can lift it to a measure $\mu'$ on $S^*\mathbb H$; then
\begin{equation}
  \label{e:mu-prime}
\mathcal T^*\mu'=\tilde\mu\times e^{-2\nu t}\,dt,
\end{equation}
where $\tilde\mu$ is some Radon measure on $(\partial \mathbb H\times
\partial \mathbb H)_\Delta$.

For each $\gamma\in \PSL(2, \mathbb R)$, we can calculate
$$
\gamma(\mathcal T(q_1,q_2,t))=\mathcal T\bigg(\gamma(q_1),\gamma(q_2),
t+{1\over 2}\log\bigg|{\gamma'_{\mathbb B}(q_1)\over\gamma'_{\mathbb B}(q_2)}\bigg|\bigg),
$$
where $\gamma'_{\mathbb B}(q)$ is the derivative of $\gamma$ considered as a
transformation on the ball model $\mathbb B$ with the identification map
$\mathbb H\to \mathbb B$ given by $z\mapsto (z-i)/(z+i)$; if
\begin{equation}
  \label{e:gamma}
\gamma(z)={az+b\over cz+d},\
\begin{pmatrix}a & b\\ c & d\end{pmatrix}\in\SL(2,\mathbb R),
\end{equation}
then
$$
|\gamma'_{\mathbb B}(q)|={q^2+1\over (aq+b)^2+(cq+d)^2}.
$$
We see then that the measure $\mu'$
defined by~\eqref{e:mu-prime} is invariant under the action of
$\Gamma$ on $S^* \mathbb B$ if and only if
for each $\gamma\in\Gamma$,
\begin{equation}
  \label{e:equivariance}
\gamma^*\tilde\mu=|\gamma'_{\mathbb B}(q_1)|^{\nu}|\gamma'_{\mathbb B}(q_2)|^{-\nu}\tilde\mu,
\end{equation}
where
$$
(\gamma^*\tilde\mu)(A)=\tilde\mu((\gamma\times\gamma)(A)),\
A\subset (\partial \mathbb B\times \partial \mathbb B)_\Delta.
$$
In particular, if $\hat\mu$ is a Radon measure on $\partial \mathbb B$ such that
for each $\gamma\in\Gamma$,
\begin{equation}
  \label{e:hat-mu-cond}
\gamma^*\hat\mu=|\gamma'_{\mathbb B}(q)|^{2\nu+1}\hat\mu,
\end{equation}
then a measure $\tilde\mu$ satisfying~\eqref{e:equivariance}
is given by (compare with~\cite[(14.14)]{b}, bearing in mind
that we use the half-plane model)
\begin{equation}
  \label{e:tilde-mu}
\tilde\mu=|q_1-q_2|^{-2(\nu+1)}|q_1+i|^{2(\nu+1)}
|q_2+i|^{2\nu}\hat\mu\times dq_2.
\end{equation}

Now, fix a cusp region $\mathcal C_j$ on $M$ and assume for simplicity
that $\infty\in \partial\mathbb H$ is a preimage of the corresponding
cusp. Let $\Gamma_\infty$ be the group of all elements of $\Gamma$ fixing
$\infty$; without loss of generality, we may assume that it is
generated by the shift $z\to z+1$. Then all the preimages of the cusp
of $\mathcal C_j$ are given by
$$
\{\gamma(\infty)\mid\gamma\Gamma_\infty\in \Gamma/\Gamma_\infty\};
$$
i.e., they are indexed by right cosets of $\Gamma_\infty$ in $\Gamma$. Note that
for $\gamma$ given by~\eqref{e:gamma},
\begin{equation}
  \label{e:ders}
\gamma(\infty)=a/c,\
|\gamma'_{\mathbb B}(\infty)|={1\over a^2+c^2}.
\end{equation}
A canonical system of coordinates on $\mathcal C_j$ is given by
\begin{equation}
  \label{e:c-q}
(r,\theta)\in (R_j,\infty)\times \mathbb S^1\mapsto \pi_\Gamma\bigg({\theta+ie^r\over 2\pi}\bigg).
\end{equation}
%
%
\begin{prop}\label{l:lift}
The lift of the measure $\mu_{j\nu}$ defined in~\eqref{e:mu-j-nu}
corresponds under~\eqref{e:mu-prime} to $(2\pi)^{2\nu+1}\tilde\mu$,
with $\tilde\mu$ given by~\eqref{e:tilde-mu}, and
\begin{equation}\label{e:hat-mu}
\hat\mu=\sum_{\gamma\Gamma_\infty\in\Gamma/\Gamma_\infty}
{\delta_{a/c}\over (a^2+c^2)^{2\nu +1}};
\end{equation}
here $\delta$ denotes a delta measure. (Note that the values $a/c$ are
distinct for different cosets, as $\Gamma_\infty$ is the stabilizer of
$\infty$.)
\end{prop}
\begin{proof}
The measure $\hat\mu$ is well-defined, as the series
\begin{equation}
  \label{e:series}
\sum_{\gamma\Gamma_\infty\in\Gamma/\Gamma_\infty}{1\over (a^2+c^2)^{2\nu+1}}=
\sum_{\gamma\Gamma_\infty\in\Gamma/\Gamma_\infty}(\Imag\gamma^{-1}(i))^{2\nu+1}
\end{equation}
converges, by convergence of Eisenstein series~\eqref{e:eis-actual}.
By~\eqref{e:ders}, the measure $\hat\mu$
satisfies~\eqref{e:hat-mu-cond}; therefore, it produces a
measure~$\mu$ supported on the cosphere bundle $S^*M$ and
satisfying~\eqref{e:expander}.  Moreover, since $\hat\mu$ is supported
on the set of the preimages of the cusp of $\mathcal C_j$, $\mu$ is
supported on $\mathcal A^-_j$. It then suffices to study the
restriction of $\mu$ to $\widehat{\mathcal A}^-_j$. To this end, take
$A\subset (R_j,\infty)\times \mathbb S^1$ and consider
$$
\widetilde A=\{(r,\theta;-1,0)\in T^*\mathcal C_j\mid
(r,\theta)\in A\}\subset \widehat{\mathcal A}^-_j.
$$
Since
$$
\mathcal T(\infty,q,t)=\bigg(q+i|i+q|e^{-t},-i{e^t\over |i+q|}\bigg),
$$
we get
$$
\begin{gathered}
\widetilde A=\pi_\Gamma\mathcal T(\{(\infty,q,t)\mid (q,t)\in\check A\}),\\
\check A=\{(\theta/(2\pi),-r+\ln(2\pi)+\ln|i+\theta/(2\pi)|)\mid (r,\theta)\in A\}. 
\end{gathered}
$$
Then
$$
\begin{gathered}
\mu(\widetilde A)
=\int_{(q,t)\in\check A} |i+q|^{2\nu}e^{-2\nu t}\,dqdt
=(2\pi)^{-2\nu-1}\int_A e^{2\nu r}\,drd\theta
\end{gathered}
$$
and the proof is finished by the definition~\eqref{e:mu-j-nu} of $\mu_{j\nu}$.
\end{proof}
%
%
In particular, for the modular surface the measure $\hat\mu$
is given by
$$
\hat\mu=\sum_{m,n\in \mathbb Z\atop n\geq 0,\ m\perp n}
{\delta_{m/n}\over (m^2+n^2)^{2\nu+1}}. 
$$
Finally, we note that for $\nu>1/2$ one can prove Theorem~\ref{t:main}
for hyperbolic surfaces using the series representation for the
Eisenstein function
\begin{equation}
  \label{e:eis-actual}
\widetilde E(\lambda;z)
=(2\pi)^{1/2-i\lambda}\sum_{\Gamma_\infty\gamma\in\Gamma_\infty\backslash\Gamma}
(\Imag\gamma(z))^{1/2-i\lambda},\
z\in \mathbb H.
\end{equation}
This series converges absolutely~\cite[Theorem~2.1.1]{k}; since it is
invariant under $\Gamma$ and each of its terms solves~\eqref{e:main}
on $\mathbb H$, it gives rise to a solution~$\widehat E(\lambda,z)$
of~\eqref{e:main}.  It can also be seen that~\eqref{e:eis-actual}
converges in $L^2$ of a fundamental domain of $\Gamma$, if we take out
the term with $\gamma=1$; therefore, $\widehat E(\lambda,z)$
satisfies~\eqref{e:eisenstein-bd} and we have
$$
\widetilde E(\lambda;z)=E_j(\lambda;\pi(z)).
$$
The $(2\pi)^{1/2-i\lambda}$ prefactor here is due to the fact that
in our normalization of the Eisenstein series, the incoming term is
given by $e^{(1/2-i\lambda)r}$, and a canonical system of coordinates
is given by~\eqref{e:c-q}, with a $2\pi$ factor there since the stabilizer
of the cusp is generated by the transformation $z\mapsto z+1$,
while we need the $\{r=0\}$ circle to have length $2\pi$.

One then proceeds as in the proof of~\cite[Theorem~2]{g-n}, by
analysing the microlocal limit of each term of the Eisenstein series
and showing that the off-diagonal terms of the sum $\langle A
\widetilde E,\widetilde E\rangle$ are negligible. (The analysis
of~\cite{g-n} is dramatically simplified, as we are not asking for an
estimate on the remainder and thus one can sum over an $h$-independent
number of the elements of the group in~\cite[Lemma~7]{g-n} and use
standard microlocal analysis.)

\noindent\textbf{Acknowledgements.} I would like to thank Maciej Zworski for
his interest in the project and helpful advice.  I would also like to
thank Steven Zelditch for several helpful discussions on the nature of
measures satisfying~\eqref{e:expander} and in particular directing me
to the formula~\eqref{e:tilde-mu}, and St\'ephane Nonnenmacher and
Yves Colin de Verdi\`ere for many suggestions for improving the
manuscript. I am especially thankful to an anonymous referee for
reading the paper carefully and many useful remarks. Finally, I am
grateful for partial support from NSF grant DMS-0654436.



\begin{thebibliography}{0}
\bibitem[Bo]{b} D.~Borthwick,
{\em Spectral theory of infinite-area hyperbolic surfaces,\/}
Birkh\"auser, 2007.

\bibitem[CdV]{cdvG} Y.~Colin de Verdi\`ere,
{\em Pseudo-Laplaciens II,\/}
Ann. Inst. Fourier \textbf{33}(1983), 87--113.

\bibitem[CdV2]{cdv} Y.~Colin de Verdi\`ere,
{\em Ergodicit\'e et fonctions propres du Laplacien,\/}
Comm. Math. Phys. \textbf{102}(1985), 497--502.

\bibitem[DaVa]{d-v} K.~Datchev and A.~Vasy,
{\em Propagation through trapped sets and semiclassical resolvent
estimates,\/}
preprint, \href{http://arxiv.org/abs/1010.2190}{arXiv:1010.2190}.

\bibitem[DeYo]{d-y} M.F.~Demers and~L.-S.~Young,
{\em Escape rates and conditionally invariant measures,\/}
Nonlinearity \textbf{19}(2006), 377--397.

\bibitem[Dya]{zeeman} S.~Dyatlov, 
{\em Asymptotic distribution of quasi-normal modes for Kerr--de Sitter black holes,\/}
to appear in Ann. Henri Poincar\'e, preprint,
\href{http://arxiv.org/abs/1101.1260}{arXiv:1101.1260}.

\bibitem[G\'eLe]{g-l} P.~G\'erard and \'E.~Leichtnam,
{\em Ergodic properties of eigenfunctions for the Dirichlet problem,\/}
Duke Math. J. \textbf{71}(1993), 559--607.

\bibitem[GrSj]{gr-s} A.~Grigis and J.~Sj\"ostrand,
{\em Microlocal analysis for differential operators: an introduction,\/}
Cambridge University Press, 1994.

\bibitem[GuiNa]{g-n} C.~Guillarmou and F.~Naud,
{\em Equidistribution of Eisenstein series on convex co-compact hyperbolic manifolds,\/}
preprint, \href{http://arxiv.org/abs/1107.2655}{arXiv:1107.2655}.

\bibitem[GuiSt]{g-s} V.~Guillemin and S.~Sternberg,
{\em Geometric asymptotics,\/}
AMS, 1990.

\bibitem[H\"oIII]{ho3} L.~H\"ormander,
{\em The Analysis of Linear Partial Differential Operators III.
Pseudo-differential Operators,\/}
Springer, 1985.

\bibitem[H\"oIV]{ho4} L.~H\"ormander,
{\em The Analysis of Linear Partial Differential Operators IV.
Fourier Integral Operators,\/}
Springer, 1985.

\bibitem[Ja]{jak} D.~Jakobson,
{\em Quantum unique ergodicity for Eisenstein series on $\PSL_2(\mathbb Z)\backslash \PSL_2(\mathbb R)$,\/}
Ann. Inst. Fourier \textbf{44}(1994), 1477--1504.

\bibitem[KeNoNoSi]{n2} J.P.~Keating, S.~Nonnenmacher, M.~Novaes, M.~Sieber,
{\em On the resonance eigenstates of an open quantum baker map,\/}
Nonlinearity \textbf{21} (2008), 2591--2624.

\bibitem[Ku]{k} T.~Kubota,
{\em Elementary theory of Eisenstein series,\/}
J.~Wiley, 1973.

\bibitem[Li]{lin} E.~Lindenstrauss,
{\em Invariant measures and arithmetic quantum unique ergodicity,\/}
Ann. of Math. \textbf{163}(2006), 165--219.

\bibitem[LuSa]{l-s} W.~Luo and P.~Sarnak,
{\em Quantum ergodicity of eigenfunctions on $\PSL_2(\mathbb Z)\backslash \mathbb H^2$,\/}
Pub. Math. de l'IHES \textbf{81}(1995), 207--237.

\bibitem[M\"u]{mu} W.~M\"uller,
{\em Spectral geometry and scattering theory for certain
complete surfaces of finite volume,\/}
Invent. Math. \textbf{109}(1992), 265--305.

\bibitem[No]{n} S.~Nonnenmacher,
{\em Spectral problems in open quantum chaos,\/}
preprint, \href{http://arxiv.org/abs/1105.2457}{arXiv:1105.2457}.

\bibitem[NoRu]{n1} S.~Nonnenmacher and M.~Rubin,
{\em Resonant eigenstates in quantum chaotic scattering,\/}
Nonlinearity \textbf{20}(2007), 1387--1420.

\bibitem[NoZw]{n-z} S.~Nonnenmacher and M.~Zworski,
{\em Quantum decay rates in chaotic scattering,\/}
Acta Math. \textbf{203}(2009), 149--233.

\bibitem[PeRaRi]{p-r-r} Y.N.~Petridis, N.~Raulf, and M.S.~Risager,
{\em Quantum limits of Eisenstein series and scattering states,\/}
preprint, \href{http://arxiv.org/abs/1111.6615}{arXiv:1111.6615}.

\bibitem[PhSa]{p-s} R.S.~Phillips and P.~Sarnak,
{\em On cusp forms for co-finite subgroups of $\PSL(2,\mathbb R)$,\/}
Invent. Math. \textbf{80}(1985), 339--364.

\bibitem[Sa]{sar} P.~Sarnak,
{\em Recent progress on the quantum unique ergodicity conjecture,\/}
Bull. Amer. Math. Soc. \textbf{48}(2011), 211--228.

\bibitem[Sch]{sch} A.I.~Schnirelman,
{\em Ergodic properties of eigenfunctions,\/}
Usp. Mat. Nauk. \textbf{29}(1974), 181--182.

\bibitem[So]{sound} K.~Soundararajan,
{\em Quantum unique ergodicity for~$\SL_2(\mathbb Z)\backslash\mathbb H$,\/}
Ann. of Math. \textbf{172}(2010), 1529--1538.

\bibitem[Ti]{tit} E.C.~Titchmarsh,
{\em The theory of the Riemann zeta-function,\/}
second edition, Oxford University Press, 1986.

\bibitem[Va]{v} A.~Vasy,
{\em Microlocal analysis of asymptotically hyperbolic spaces
and high energy resolvent estimates,\/}
preprint, \href{http://arxiv.org/abs/1104.1376}{arXiv:1104.1376}.

\bibitem[V\~uNg]{svn} San V\~u Ng\d oc,
{\em Syst\`emes int\'egrables semi-classiques: du local au global,\/}
Panoramas et Synth\`eses 22, 2006.

\bibitem[Ze1]{z1} S.~Zelditch,
{\em Uniform distribution of eigenfunctions on compact hyperbolic surfaces,\/}
Duke Math. J. \textbf{55}(1987), 919--941.

\bibitem[Ze2]{z2} S.~Zelditch,
{\em Mean Lindel\"of hypothesis and equidistribution of cusp forms and Eisenstein series,\/}
J. Funct. Anal. \textbf{97}(1991), 1--49.

\bibitem[Ze3]{z3} S.~Zelditch,
{\em Recent developments in mathematical quantum chaos,\/}
Curr. Dev. Math. 2009, 115--204.

\bibitem[ZeZw]{z-z} S.~Zelditch and M.~Zworski,
{\em Ergodicity of eigenfunctions for ergodic billiards,\/}
Comm. Math. Phys. \textbf{175}(1996), 673--682.

\bibitem[Zw]{e-z} M.~Zworski,
{\em Semiclassical analysis,\/}
to appear in Graduate Studies in Mathematics, AMS, 2012,
\url{http://math.berkeley.edu/~zworski/semiclassical.pdf}.

\end{thebibliography}
\end{document}